\RequirePackage{xcolor}
\documentclass[journal,twoside,web]{ieeecolor}
\usepackage{lcsys}
\usepackage[utf8]{inputenc}

\usepackage{etoolbox}
\makeatletter
\@ifundefined{color@begingroup}%
  {\let\color@begingroup\relax
   \let\color@endgroup\relax}{}%
\def\fix@ieeecolor@hbox#1{%
  \hbox{\color@begingroup#1\color@endgroup}}
\patchcmd\@makecaption{\hbox}{\fix@ieeecolor@hbox}{}{\FAILED}
\patchcmd\@makecaption{\hbox}{\fix@ieeecolor@hbox}{}{\FAILED}

\usepackage{booktabs}

\usepackage{amsmath} 
\usepackage{bbm}
\usepackage{bm}
\usepackage{amssymb}
\usepackage{mathtools}

\usepackage{amsthm}

\theoremstyle{definition}
\newtheorem{definition}{Definition}
\newtheorem{theorem}{Theorem}
\newtheorem{lemma}{Lemma}
\newtheorem{assumption}{Assumption}

\newtheorem{problem}{Problem}
\newtheorem{proposition}{Proposition}

\usepackage{hyperref}
\hypersetup{
    colorlinks=true,
    linkcolor=black,
    filecolor=black,      
    urlcolor=blue,
    citecolor=blue,
    }
\usepackage{cite}

\usepackage{todonotes}

\usepackage{blindtext}

\title{
Constraint-Driven Optimal Control for Emergent Swarming and Predator Avoidance
}

\author{Logan E. Beaver, \emph{Student Member, IEEE},  Andreas A. Malikopoulos, \emph{Senior Member, IEEE} 
	\thanks{L.E. Beaver and A.A. Malikopoulos are with the Department of Mechanical Engineering, University of Delaware, Newark, DE, USA (emails: lebeaver@udel.edu, andreas@udel.edu).}%
}

\begin{document}

\maketitle

\thispagestyle{empty}

\begin{abstract}
In this letter, we present a constraint-driven optimal control framework that achieves emergent cluster flocking within a constrained 2D environment.
We formulate a decentralized optimal control problem that includes safety, flocking, and predator avoidance constraints. We explicitly derive conditions for constraint compatibility and propose an event-driven constraint relaxation scheme, which we map to an equivalent finite state machine that intuitively describes the behavior of each agent in the system.
Instead of minimizing control effort, as it is common in the ecologically-inspired robotics literature, in our approach, we minimize each agent's deviation from their most efficient locomotion speed.
Finally, we demonstrate our approach in simulation both with and without the presence of a predator.
\end{abstract}

\begin{IEEEkeywords} 
biologically-inspired methods, complex systems, multi-agent systems, optimal control, optimization
\end{IEEEkeywords}

\section{Introduction}

\IEEEPARstart{M}{ulti-agent} systems have attracted considerable attention in many applications due to their natural parallelization, general adaptability, and ability to self-organize \cite{Oh2017}.
One emerging application of multi-agent systems is mimicking the aggregate motion of certain birds and fish, also known as cluster flocking or swarming \cite{Beaver2020AnFlockingb}.
There are several purported advantages of cluster flocking in biological systems, including predator avoidance and estimating total population \cite{Bajec2009OrganizedBirds}.

In this letter, we derive a distributed control algorithm that induces cluster flocking in a multi-agent system.
Prior work has primarily relied on reinforcement learning to achieve predator avoidance, including a multi-level approaches \cite{La2015MultirobotAvoidance} and policy sharing \cite{Morihiro2006EmergenceLearning,Hahn2019EmergentLearning}.
Traditional control approaches tend to achieve swarming behavior by implementing Reynolds flocking rules using potential fields \cite{Beaver2020AnFlockingb}.
These approaches have two major drawbacks. First, they inevitably drive agents into a regular lattice formation \cite{Olfati-Saber2006FlockingTheory}, which is not conductive to swarming. Second, potential fields are known to cause steady oscillation in agent trajectories and exacerbate deadlock in constrained environments \cite{Koren1991PotentialNavigation}.

In contrast to existing approaches, we propose a biologically-inspired approach based on an analysis of sand-eel schools in the presence of predators \cite{Pitcher1983Predator-avoidanceSplit}.
In this letter, we build on our previous work with set-theoretic control \cite{Beaver2020BeyondFlocking,Beaver2020Energy-OptimalConstraints,Beaver2020AnFlocking}, where we embed inter-agent and environmental interactions as state and control constraints in an optimal control problem.
Our set-theoretic approach has the advantage of being interpretable, i.e., the cause of an agents' action can be deduced by examining which constraints are currently active.
Our technical results are closely related to the control barrier function (CBF) literature, particularly multi-agent CBFs \cite{Wang2017SafetySystems}.
However, our approach does not require the constraints to be encoded as sub-level sets of a continuous function\textemdash we work with the sets directly.
We also propose a solution to the open problem of constraint incompatibility through an event-triggered constraint relaxation scheme.
Finally, we present a mapping between constraint-driven control and finite state machines, which provides a rigorous and interpretable description of each boids' behavior.
The contributions of this letter are:
\begin{enumerate}
    \item a decentralized optimal control algorithm that yields emergent swarming behavior (Problem \ref{prb:optimalControl}),
    \item an event-triggered scheme to selectively relax constraints and guarantee feasibility (Lemmas \ref{lma:conflict}--\ref{lma:joint}),
    \item a rigorous mapping between constraint-driven control and finite state machines (Proposition \ref{prp:stateMachine}), and
    \item simulation results demonstrating emergent cluster flocking and predator avoidance behaviors (Section \ref{sec:sim}).
\end{enumerate}

The remainder of the letter is organized as follows.
In Section \ref{sec:problem}, we formulate the cluster flocking problem and discuss our working assumptions.
In Section \ref{sec:solution}, we derive our optimal control policy, derive the safe action sets, and map the problem to a finite state machine.
In Section \ref{sec:sim}, we validate our results in two simulations with $15$ boids; the first demonstrates emergent cluster flocking, and the second demonstrates predator avoidance.
Finally, we draw conclusions and propose some directions for future research in Section \ref{sec:conclusion}.

\section{Problem Formulation} \label{sec:problem}

We consider a set of $N\in\mathbb{N}$ boids indexed by the set $\mathcal{B} = \{1, 2, \dots, N\}$.
Each boid $i\in\mathcal{B}$ obeys second-order integrator dynamics,
\begin{equation} \label{eq:dynamics}
    \begin{aligned} 
    \dot{\bm{p}}_i(t) = \bm{v}_i(t), \\
    \dot{\bm{v}}_i(t) = \bm{u}_i(t),
    \end{aligned}
\end{equation}
where $\bm{p}_i(t), \bm{v}_i(t) \in \mathbb{R}^2$ correspond to the position and velocity of each boid, and $\bm{u}_i(t) \in \mathbb{R}^2$ is the control input.
We also impose the state and control constraints,
\begin{align}
    \bm{p}_i(t) \in \mathcal{P}, \label{eq:pLim} \\
    \bm{u}_i(t) \in \mathcal{U}, \label{eq:uLim}
\end{align}
where $\mathcal{P}\subset\mathbb{R}^2$ is a non-empty intersection of half-planes and $\mathcal{U} = \{u\in\mathbb{R}^2 : ||u||_{\infty} \leq u_{\max} \}$ ensures the boids' do not exceed their maximum control input. 
We employ the infinity norm to simplify our mathematical exposition; however, the norm does not impose any restrictions in our approach.

We account for interactions between boids using Voronoi tessellation \cite{Fine2013UnifyingMembers}.
Under this approach, each boid is considered the center of a Voronoi cell.
We define the Voronoi set $\mathcal{V}(t) \in \mathcal{B}\times\mathcal{B}$ to contain $(i, j)$ and $(j, i)$ when the Voronoi cells $i$ and $j$ share a common edge.
Equivalently, the set $\mathcal{V}(t)$ is the Delaunay triangulation of the boids' positions.

\begin{definition}[Voronoi Neighborhood] \label{def:neighborhood}
The neighborhood of each boid $i\in\mathcal{B}$ is the set,
\begin{equation}
    \mathcal{N}_i(t) \coloneqq \big\{ j \in \mathcal{B} ~:~ (i, j) \in \mathcal{V}(t) \big\},
\end{equation}
where boid $i$ can receive information, via communication or sensing, with any other boid $j\in\mathcal{N}_i(t)$.
\end{definition}

As with $k-$nearest neighbors, the sensing radius of a Voronoi neighborhood may grow unbounded in general.
Similar to past work \cite{Cristiani2011EffectsGroups,Beaver2020BeyondFlocking}, we do not presume the boids possess infinite sensing capabilities; rather that Definition \ref{def:neighborhood} describes the interactions between boids over their relatively small separating distances.
One potential solution is to only consider Voronoi neighbors that are within a fixed sensing range \cite{Fine2013UnifyingMembers}, although results from biology demonstrate that this is, in general, unnecessary \cite{Ballerini2008InteractionStudy}.

Our objective is to generate emergent swarming behavior, such that the boids remain close to their neighbors to avoid predators \cite{Pitcher1983Predator-avoidanceSplit,Bajec2009OrganizedBirds,Hahn2019EmergentLearning}.
To achieve an aggregate swarming motion, we implement a variation of the disk flocking constraint proposed in \cite{Beaver2020BeyondFlocking}.
First, we determine the neighborhood center for each boid $i\in\mathcal{B}$,
\begin{equation} \label{eq:center}
    \bm{c}_i(t) = \frac{1}{|\mathcal{N}_i(t)|} \sum_{j\in\mathcal{N}_i(t)} \bm{p}_j(t).
\end{equation}
Note that Definition \ref{def:neighborhood} guarantees $|\mathcal{N}_i(t)| > 0$.
We use the neighborhood center to construct the relative position vector,
\begin{equation}
    \bm{r}_i(t) \coloneqq \bm{p}_i(t) - \bm{c}_i(t).
\end{equation}
Finally, to achieve swarming, we require each boid $i$ to approach and remain within a distance $R \in \mathbb{R}_{>0}$ of the neighborhood center, i.e.,
\begin{equation*}
g_i(\bm{r}_i(t), t)
    =
    \begin{cases}
    ||\bm{r}_i(t)|| - R \quad \text{if } ||\bm{r}_i(t)|| \leq R, \\
    \frac{||\dot{\bm{r}}_i(t)||}{u_{\max}}\bm{u}_i(t)\cdot\bm{r}_i(t) + \dot{\bm{r}}_i(t)\cdot\bm{r}_i(t) \text{ o.w.}
    \end{cases}
    \leq 0.
\end{equation*}
Note that the first case is trivially satisfied, i.e., the boid must remain within the disk while inside the disk.
Thus, we write
\begin{align} \label{eq:neighborhood}
    ||\bm{r}_i(t)|| > R
    &\implies \notag\\
    &\frac{||\dot{\bm{r}}_i(t)||}{u_{\max}}\bm{u}_i(t)\cdot\bm{r} + \dot{\bm{r}}_i(t)\cdot\bm{r}_i(t) \leq 0.
\end{align}
We emphasize that our objective is not to trap boid $i$ within the disk of radius $R$ centered at $\bm{c}_i(t)$.
Instead, we expect the switching neighborhood topology and dynamic motion of $\bm{c}_i(t)$ to drive the swarming behavior.
Additionally, the form of \eqref{eq:neighborhood} is inspired by energy-saving techniques in \cite{Zhou2017DistributedOptimization}.
Note that when boid $i$ travels in the ``correct" direction, i.e., $\dot{\bm{r}}_i(t)\cdot\bm{r} < 0$, the control action $\bm{u}_i(t)$ can take some values in the same direction as $\bm{r}_i(t)$.
However, when boid $i$ is traveling in the ``wrong" direction, i.e., $\dot{\bm{r}}_i(t)\cdot\bm{r}_i(t) > 0$, the control action $\bm{u}_i(t)$ must be at least partially opposed to $\bm{r}_i(t)$ to drive boid $i$ toward $\bm{c}_i(t)$.

Next, inspired by the empirical data collected on sand-eels \cite{Pitcher1983Predator-avoidanceSplit}, we model the predator as a ball of radius $\Gamma$.
We define the relative distance vector between each boid $i$ and the predator as,
\begin{equation}
    \bm{d}_i(t) \coloneqq \bm{p}_i(t) - \bm{o}_i(t),
\end{equation}
where $\bm{o}_i(t)$ is the position of the predator at time $t$.
To ensure predator avoidance, we select a value of $\Gamma$ larger than the diameter of the predator and employ a similar constraint to repel the boids,
\begin{align} \label{eq:predator}
    ||\bm{d}_i(t)|| < \Gamma
    &\implies \notag\\
    &-\frac{||\dot{\bm{d}}_i(t)||}{u_{\max}}\bm{u}_i(t)\cdot\bm{d} - \dot{\bm{d}}_i(t)\cdot\bm{d}_i(t) \leq 0.
\end{align}
With the constraints defined, our next objective is to design an optimal control problem such that the individual boid motion generates emergent swarming behavior.
To this end, we impose the following assumptions on our system.

\begin{assumption} \label{smp:continuous}
Each boid is equipped with a low-level controller that is capable of tracking the control input.
\end{assumption}
\begin{assumption} \label{smp:perfect}
Communication and sensing between the boids occurs instantaneously and noiselessly.
\end{assumption}

We impose Assumptions \ref{smp:continuous} and \ref{smp:perfect} to simplify our analysis and understand how the system performs in the ideal case.
Assumption \ref{smp:continuous} is common for trajectory generation problems, and it can be relaxed by introducing robust control terms or a safety layer, e.g., using a control barrier function \cite{Ames2019ControlApplications}.
Similarly, Assumption \ref{smp:perfect} can be relaxed by including time delays and uncertainty explicitly in the formulation and applying stochastic \cite{Dave2020a} or robust \cite{chalaki2021RobustGP} control techniques.

\begin{assumption} \label{smp:safety}
The boids have sufficient vertical space to avoid collisions without an explicit collision-avoidance constraint.
\end{assumption}

Assumption \ref{smp:safety} is common in 2D swarming applications \cite{Hahn2019EmergentLearning,Hahn2020ForagingLearning}.
Furthermore, it has been thoroughly demonstrated that adding an extra dimension of motion can significantly reduce the likelihood of collisions \cite{Morgan2016}.

\section{Solution Approach} \label{sec:solution}

We employ \emph{gradient flow} to generate the control input for each boid.
This a gradient-based optimization technique, wherein each boids' control action is a gradient descent step.
This technique has been used successfully to control multi-agent constraint-driven systems \cite{Notomista2019Constraint-DrivenSystems,Wang2017SafetySystems,Beaver2021Constraint-DrivenStudy}.
Our motivation for gradient flow is twofold.
First, it enables the boids to immediately react to their surroundings without the computational and communication costs associated with decentralized trajectory planning \cite{Beaver2020AnFlockingb}.
Second, it allows boids to freely enter and leave the domain, e.g., due to operating constraints, mechanical failure, or predation.
Each boid determines its current control action by solving an optimal control problem.
First, we derive a set of \emph{safe control inputs} and guarantee that it satisfies recursive feasibility, i.e., it remains non-empty for all future time.
For the remainder of our exposition, we omit the explicit dependence of state variables on $t$ when no ambiguity arises.

We start with the position constraint \eqref{eq:pLim}, which is not an explicit function of the control input.
Let $k = 1, 2, \dots, M$ index the $M$ hyperplanes that define the boundary of $\mathcal{P}$.
Each hyperplane $k = 1, 2, \dots, M$ consists of a normal vector  $\hat{\bm{n}}_k\in\mathbb{R}^2$ and offset $b_k\in\mathbb{R}$; 
the signed distance to the surface of hyperplane $k$ is,
\begin{equation} \label{eq:signedDist}
    d_{i,k} = \bm{p}_i \cdot \hat{\bm{n}}_k + b_k,
\end{equation}
for boid $i\in\mathcal{B}$.
Note that our convention assumes the normal vector $\bm{n}_k$ points away from the feasible region $\mathcal{P}$.
To guarantee constraint satisfaction, we require the derivative of \eqref{eq:signedDist} to be non-positive when the constraint is active, i.e.,
\begin{equation} \label{eq:constraint}
    \bm{p}_i\cdot\hat{\bm{n}}_k + b_k = 0 \implies \bm{v}_i\cdot\hat{\bm{n}}_k \leq 0.
\end{equation}
This safety constraint \eqref{eq:constraint} can be achieved by using a stopping distance constraint for each $k = 1, 2, \dots, M$ \cite{Beaver2021Constraint-DrivenStudy},
\begin{align} \label{eq:barrier}
    g_{ik} = \Big( \bm{p}_i\cdot\hat{\bm{n}}_k + b_k  \Big) + \alpha \frac{\Big( \bm{v}_i\cdot\bm{n}_k \Big)^2}{2 u_{\max}} \leq 0,
\end{align}
where $\alpha\in\mathbb{R}_{>0}$ is a parameter that determines the stopping distance.
Note that \eqref{eq:barrier} trivially satisfies \eqref{eq:constraint}.
This leads to our definition of the safe action set.
\begin{definition}[Safe Action Set] \label{def:safeSet}
For each boid $i\in\mathcal{B}$ at time $t$, the safe action set is,
\begin{equation}
\begin{aligned}
    \mathcal{A}_i^s(t) \coloneqq \Big\{ & \bm{u}_i \in \mathbb{R}^2 ~:~ ||\bm{u}_i||_{\infty} - u_{\max} \leq 0, \\
    & \Big(\bm{p}_i\cdot\hat{\bm{n}}_k + b_k \Big) + \alpha\frac{\Big(\bm{v}_i\cdot\hat{\bm{n}}_k \Big)^2}{2 u_{\max}} \leq 0, \notag\\ & \quad\quad \forall k = 1, 2, \dots, M \Big\}.
\end{aligned}
\end{equation}
\end{definition}

In our approach, we constrain the boids to remain within an axis-aligned rectangular domain, i.e., $\mathcal{P}$ is constructed from two pairs of parallel hyperplanes that intersect at right angles.
Next, we present a result that guarantees recursive feasibility for the safe action set.

\begin{theorem}\label{thm:recursion}

If a boid $i\in\mathcal{B}$ satisfies \eqref{eq:barrier} at some time $t$ for a rectangular domain $\mathcal{P}$, then any $\alpha \geq 1$ guarantees recursive feasibility of $\mathcal{A}_i^s(t)$ for all future time.
\end{theorem}

\begin{proof}
Let the rectangular domain $\mathcal{P}$ consists of four hyperplanes, indexed by $k = 1, 2, 3, 4$ such that $\hat{\bm{n}}_1 = -\hat{\bm{n}}_3$ and $\hat{\bm{n}}_2 = -\hat{\bm{n}}_4$.
Without loss of generality, let $\bm{v}_i\cdot\hat{\bm{n}}_1 > 0$ and $\bm{v}_i\cdot\hat{\bm{n}}_2 > 0$.
Taking the derivative of \eqref{eq:barrier} and combining terms yields,
\begin{equation} \label{eq:gDot1}
    \dot{g}_{ik} = \bm{v}_i\cdot\hat{\bm{n}}_k \Big( 1 + \frac{\alpha}{u_{\max}}\big(\bm{u}_i\cdot\hat{\bm{n}}_k\big)\Big).
\end{equation}
We seek a control input such that $\dot{g}_{ik} \leq 0$.
For $k = 1,2$, dividing by $\bm{v}_i\cdot\hat{\bm{n}}_k > 0$ yields a condition on $\bm{u}_i$,
\begin{align} \label{eq:un12}
    \bm{u}_i\cdot\hat{\bm{n}}_{1} \leq - \frac{u_{\max}}{\alpha}, \quad \bm{u}_i\cdot\hat{\bm{n}}_{2} \leq - \frac{u_{\max}}{\alpha}.
\end{align}
Similarly, for $k = 3, 4$, dividing by $\bm{v}_i\cdot\hat{\bm{n}}_k < 0$ implies,
\begin{align} \label{eq:un34}
    \bm{u}_i\cdot\hat{\bm{n}}_{3} \geq - \frac{u_{\max}}{\alpha}, \quad     \bm{u}_i\cdot\hat{\bm{n}}_{4} \geq - \frac{u_{\max}}{\alpha}.
\end{align}
Substituting $\hat{\bm{n}}_{1} = -\hat{\bm{n}}_{3}$ and $\hat{\bm{n}}_2 = -\hat{\bm{n}}_4$ into \eqref{eq:un34} yields the conditions,
\begin{equation} \label{eq:un342}
    \bm{u}_i\cdot\hat{\bm{n}}_{1} \leq \frac{u_{\max}}{\alpha}, \quad     \bm{u}_i\cdot\hat{\bm{n}}_{2} \leq \frac{u_{\max}}{\alpha}.
\end{equation}
Thus, to guarantee $g_{i,k}$ is nonincreasing, the control input must satisfy \eqref{eq:un12} and \eqref{eq:un342}, i.e.,
\begin{align}
    \bm{u}_{i}\cdot\hat{\bm{n}}_1 \leq - \frac{u_{\max}}{\alpha} \leq \frac{u_{\max}}{\alpha}, \\
    \bm{u}_{i}\cdot\hat{\bm{n}}_2 \leq - \frac{u_{\max}}{\alpha} \leq \frac{u_{\max}}{\alpha}.
\end{align}
This is satisfied by the candidate control action,
\begin{equation} \label{eq:thm1Control}
    \bm{u}_i = -\frac{u_{\max}}{\alpha}\hat{\bm{n}}_1 -\frac{u_{\max}}{\alpha}\hat{\bm{n}}_2,
\end{equation}
as $\hat{\bm{n}}_1 \cdot \hat{\bm{n}}_2 = 0$ by definition.
In our axis-aligned domain, the control constraint implies,
\begin{equation}
    ||\bm{u}_i||_{\infty} = \max\Bigg\{\frac{u_{\max}}{\alpha},  \frac{u_{\max}}{\alpha}\Bigg\} = \frac{1}{\alpha} u_{\max},
\end{equation}
which satisfies \eqref{eq:uLim} for $\alpha \geq 1$.
Finally, for the case that $\bm{v}_i\cdot\hat{\bm{n}}_k = 0$ for any $k = 1, 2, 3, 4$, the corresponding derivative $\dot{\bm{g}}_{ik} = 0$ for every control input.
\end{proof}

Thus, given a feasible initial state, Theorem \ref{thm:recursion} guarantees that each boid's trajectory will remain feasible indefinitely if its control action is selected from $\mathcal{A}_i^s(t)$.
In conjunction with the safe action set, we also impose the swarming and predator avoidance constraints, \eqref{eq:neighborhood} and \eqref{eq:predator}, to achieve emergent cluster flocking behavior.
The following results provide the conditions for constraint incompatibility, i.e., under what conditions the set of feasible control actions becomes empty.

\begin{lemma} \label{lma:conflict}
For a boid $i\in\mathcal{A}$, let $k = 1, 2$ index two perpendicular hyperplanes in the rectangular domain such that $\bm{v}_i\cdot\hat{\bm{n}}_k \geq 0$.
Then, if \eqref{eq:barrier} is strictly equal to zero and $||\bm{r}_i|| > R_i$, there is no feasible action if none of the conditions,
\begin{equation}
||\dot{\bm{r}}_i||\Bigg(
\begin{bmatrix}
1 \\ \alpha^{-1} \\ 1 \\ \alpha^{-1}
\end{bmatrix}
    (\hat{\bm{n}}_1\cdot\hat{\bm{r}}_i) + 
\begin{bmatrix}
1 \\ 1 \\ \alpha^{-1} \\ \alpha^{-1}
\end{bmatrix}
(\hat{\bm{n}}_2\cdot\hat{\bm{r}}_i)\Bigg) 
    \geq \dot{\bm{r}}_i\cdot\hat{\bm{r}}_i,
\end{equation}
hold at time $t$ for $k = 1, 2$.
\end{lemma}

\begin{proof}
Under the premise of Lemma \ref{lma:conflict}, we must determine when the control constraints,
\begin{align}
    \frac{||\bm{v}_i||}{u_{\max}}\bm{u}_i\cdot\hat{\bm{r}}_i + \dot{\bm{r}}_i\cdot\hat{\bm{r}}_i &\leq 0, \label{eq:lma1} \\
    \bm{v}_i\cdot\hat{\bm{n}}_k\big(1 + \frac{\alpha}{u_{\max}}(\bm{u}_i\cdot\hat{\bm{n}}_k)\big) &\leq 0, \label{eq:lma1b}
\end{align}
are incompatible.
First, $\dot{\bm{r}}_i = 0$ satisfies \eqref{eq:lma1} for any $\bm{u}_i$, thus, we may divide \eqref{eq:lma1} by $||\dot{\bm{r}}_i||$ and work with unit vectors for the remainder of the proof, i.e.,
\begin{equation} \label{eq:lma1a}
    \frac{1}{u_{\max}}\bm{u}_i\cdot\hat{\bm{r}}_i + \hat{\dot{\bm{r}}}_i\cdot\hat{\bm{r}}_i \leq 0.
\end{equation}
Next, we consider the control $\bm{u}_i = -u_1\hat{\bm{n}}_1 - u_2\hat{\bm{n}}_2$.
From the proof of Theorem \ref{thm:recursion}, \eqref{eq:uLim} and \eqref{eq:un12} imply that $u_1$ and $u_2$ must satisfy,
\begin{equation} \label{eq:uBounds}
    1 \geq \frac{u_1}{u_{\max}} \geq \frac{1}{\alpha}, \quad 1\geq \frac{u_2}{u_{\max}} \geq \frac{1}{\alpha}.
\end{equation}
The swarming constraint \eqref{eq:lma1a} becomes,
\begin{equation} \label{eq:lm1Inequality}
    \frac{u_1}{u_{\max}}(\hat{\bm{n}}_1\cdot\hat{\bm{r}}_i) + \frac{u_2}{u_{\max}}(\hat{\bm{n}}_2\cdot\hat{\bm{r}}_i)  \geq \hat{\dot{\bm{r}}}_i\cdot\hat{\bm{r}}_i.
\end{equation}
The result follows from substituting the bounds \eqref{eq:uBounds} into \eqref{eq:lm1Inequality}.
\end{proof}

\begin{lemma} \label{lma:flash}
For a boid $i\in\mathcal{A}$, let $k = 1, 2$ index two perpendicular hyperplanes in the rectangular domain such that $\bm{v}_i\cdot\hat{\bm{n}}_k \geq 0$.
Then, if \eqref{eq:barrier} is strictly equal to zero and $||\bm{d}_i|| > \Gamma$, there is no feasible action if none of the conditions,
\begin{equation}
||\dot{\bm{d}}_i||\Bigg(
\begin{bmatrix}
1 \\ \alpha^{-1} \\ 1 \\ \alpha^{-1}
\end{bmatrix}
    (\hat{\bm{n}}_1\cdot\hat{\bm{d}}_i) + 
\begin{bmatrix}
1 \\ 1 \\ \alpha^{-1} \\ \alpha^{-1}
\end{bmatrix}
(\hat{\bm{n}}_2\cdot\hat{\bm{d}}_i)\Bigg) 
    \leq \dot{\bm{d}}_i\cdot\hat{\bm{d}}_i,
\end{equation}
hold at time $t$ for $k = 1, 2$.
\end{lemma}

\begin{proof}
  The proof Lemma \ref{lma:flash} is identical to Lemma \ref{lma:conflict}, and thus we omit it.
\end{proof}

\begin{lemma} \label{lma:joint}
For a boid $i\in\mathcal{A}$, let $k = 1, 2$ index two perpendicular hyperplanes in the rectangular domain such that $\bm{v}_i\cdot\hat{\bm{n}}_k \geq 0$.
Then, if \eqref{eq:barrier} is strictly equal to zero, $||\bm{r}_i|| > R_i$, and and $||\bm{d}_i|| > \Gamma$, there is no feasible control action if the linear system of equations,
\begin{equation*}
    \begin{bmatrix}
    ||\dot{\bm{r}}_i||\hat{\bm{n}}_1\cdot\hat{\bm{r}}_i &
    ||\dot{\bm{r}}_i||\hat{\bm{n}}_2\cdot\hat{\bm{r}}_i \\
     -||\dot{\bm{d}}_i||\hat{\bm{n}}_1\cdot\hat{\bm{d}}_i &
     -||\dot{\bm{d}}_i||\hat{\bm{n}}_2\cdot\hat{\bm{d}}_i 
    \end{bmatrix}
    \begin{bmatrix}
    \frac{u_1}{u_{\max}} \\ \frac{u_2}{u_{\max}}
    \end{bmatrix}
    \geq 
    \begin{bmatrix}
    \dot{\bm{r}}_i\cdot\hat{\bm{r}}_i \\ -\dot{\bm{d}}_i\cdot\hat{\bm{d}}_i 
    \end{bmatrix}
\end{equation*}
has no solution that also satisfies $\frac{1}{\alpha} \leq \frac{u_1}{u_{\max}} \leq 1$ and $\frac{1}{\alpha} \leq \frac{u_2}{u_{\max}} \leq 1$.
\end{lemma}

\begin{proof}
 The proof of Lemma \ref{lma:joint} is constructed by satisfying Lemmas \ref{lma:conflict} and \ref{lma:flash} jointly.
\end{proof}

Generally, the problem of constraint incompatibility has been solved in the ecologically-inspired robotics literature by introducing slack variables \cite{Egerstedt2018RobotAutonomy,Ibuki2020Optimization-BasedBodies}.
However, it is unclear why one would add slack to the predator avoidance constraint when the premise of Lemma \ref{lma:flash}  is not satisfied.
For this reason, we use Lemmas \ref{lma:conflict}--\ref{lma:joint} to construct a finite state machine that completely describes the behavior of each boid. 
Note that defining an appropriate evasive behavior when Lemma \ref{lma:flash} holds, e.g., a fountain \cite{Berlinger2021Self-OrganizedCollective} or flash \cite{Pitcher1983Predator-avoidanceSplit} maneuver, is beyond the scope of this work; in our simulations (Section \ref{sec:sim}), we simply relax the predator-avoidance constraint.

\begin{proposition} \label{prp:stateMachine}
Each boid $i\in\mathcal{B}$ can be modeled as a finite state machine with three states: 1) \emph{Nominal}, which considers the safety, swarming, and predator avoidance constraints; 2) \emph{Strained}, which relaxes the swarming constraint; and 3) \emph{Evasive}, where the boid executes an evasive maneuver.
Boid $i$ transitions between these states based on whether the premises of Lemmas \ref{lma:conflict}--\ref{lma:joint} are satisfied at each time; this is presented in Fig. \ref{fig:stateMachine}.

\begin{figure}[ht]
    \centering
    \includegraphics[width=0.9\linewidth]{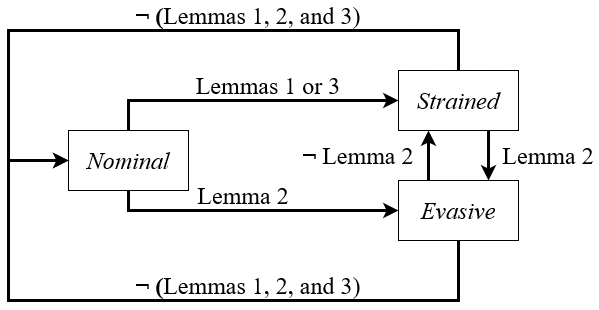}
    \caption{A finite state machine that describes each boids' feasible action space based on whether the premise of Lemmas \ref{lma:conflict}--\ref{lma:joint} are satisfied.}
    \label{fig:stateMachine}
\end{figure}

\end{proposition}

With the feasible action set and finite state machine defined,  each  boid $i\in\mathcal{B}$ also requires a notion of performance to select the ``best" control input.
Following the ecologically-inspired paradigm \cite{Egerstedt2018RobotAutonomy} would suggest minimizing the norm of the control input; this arguably yields a minimum effort policy.
However, we have previously demonstrated that selecting an appropriate objective function is critical to achieve a desired emergent behavior \cite{Beaver2021Constraint-DrivenStudy}.
As discussed in \cite{Pitcher1983Predator-avoidanceSplit}, sand-eels tend to cruise at a constant speed of approximately $2$ body lengths per second.
Thus, we require each boid to match an optimal swimming speed, denoted $||\bm{v}_i^*||$, as closely as possible, i.e.,
\begin{equation} \label{eq:cost}
    J_i\big(\bm{v}_i(t)\big) = \Big( ||\bm{v}_i(t)|| - ||\bm{v}^*||\Big)^2.
\end{equation}
We interpret the optimal swimming speed as being bio-mechanically advantageous, i.e., if $J = ||\bm{u}_i||$ minimizes energy consumption, then \eqref{eq:cost} corresponds to minimum-power locomotion.
Combining the cost \eqref{eq:cost} with the state-machine architecture outlined in Proposition \ref{prp:stateMachine} yields the optimal control problem solved by each boid.

\begin{problem} \label{prb:optimalControl}
For each boid $i\in\mathcal{B}$ at time $t$, apply the control action that solves,
\begin{align*}
    &\min_{\bm{u}_i(t)} \, \Big( ||\bm{v}_i(t)|| - ||\bm{v}^*||\Big)^2 \\
    \text{subject to:}&\\
    &\bm{u}_i(t) \in\mathcal{A}_i^s\big(t),\, \eqref{eq:dynamics},\, \eqref{eq:neighborhood},\, \eqref{eq:predator},
\end{align*}
where \eqref{eq:neighborhood} and \eqref{eq:predator} are relaxed according to Proposition \ref{prp:stateMachine}.
\end{problem}
The final step is to tune the system parameters, which we discuss, along with the simulation results, in the following section.

\section{Simulation} \label{sec:sim}

To validate our optimal control policy, we solved Problem \ref{prb:optimalControl} for $N = 15$ boids over a $120$ second time interval.
Next, we present our simulation parameters and the physical intuition behind them, followed by simulations that demonstrate the desired cluster flocking and predator avoidance behaviors.
Additional details and simulation videos can be found on the dedicated website of manuscript, \url{https://sites.google.com/view/ud-ids-lab/swarming}.

Based on the information given in \cite{Pitcher1983Predator-avoidanceSplit}, we selected a diameter of $5$ cm for each boid, which implies an optimal speed of approximately $12.5$ cm/s.
Intuitively, it is desirable for each boid $i\in\mathcal{B}$ to have a small actuation limit relative to the desired speed. Each boid ought to approach its neighborhood center $\bm{c}_i$ at a high speed, overshoot it, and circle back toward $\bm{c}_i$ in a wide arc.
This circling motion will also influence the topology of the Voronoi neighborhoods, which will further perturb the flock.
Ideally, these perturbations will push some boids to the edge of the flock to counteract flock collapse \cite{Olfati-Saber2006FlockingTheory}.
Additionally, we select a square domain $\mathcal{P}$ that is large enough for cluster flocking to occur.
We summarize our simulation parameters in Table \ref{tab:parameters}.

\begin{table}[ht]
    \centering
    \caption{Simulation parameters used to generate swarming behavior.}
    \begin{tabular}{cccccc}
        Domain Length (m) & $v^*$ (m/s) & $u_{\max}$ (m/s$^2$) & $R$ (cm) & $\Gamma$ (cm) \\
        \midrule
        6 & 0.125 & 0.1 & 2.5 & 25
    \end{tabular}
    \label{tab:parameters}
\end{table}

To simulate the swarming behavior, we initialize all boids at rest with random initial positions within the domain $\mathcal{P}$ such that none overlap.
At each time step, we solve Problem \ref{prb:optimalControl} and may relax some constraints according to Proposition \ref{prp:stateMachine}.
The behavior of the swarm is visualized in Figs. \ref{fig:swarming1} and \ref{fig:swarming2}, which show two time snapshots from the simulation.
Figure \ref{fig:swarming1} shows the initial behavior of the boids $21$ seconds into the simulation.
Starting near the center of the domain, the boids begin travelling in the north-western direction and exhibit a swirling motion.
This is is visible from their tails, which show $8$ seconds of trajectory history.
After reaching the north-western hyperplane, the boids quickly turn around and travel to the south-east.
Figure \ref{fig:swarming2} shows behavior qualitatively similar to the cruising behavior described by \cite{Pitcher1983Predator-avoidanceSplit}, where their velocities are relatively constant in direction and magnitude.

\begin{figure}[ht]
\centering
\includegraphics[width=\linewidth]{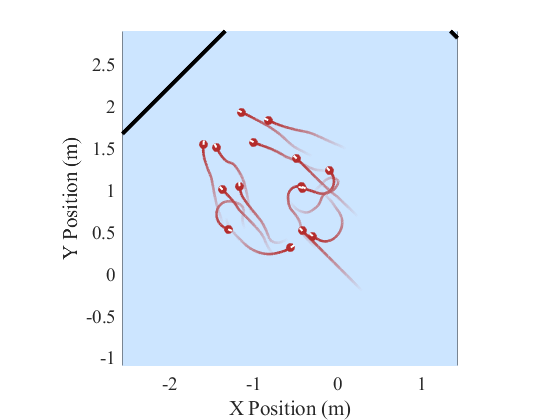}
\caption{Boids circling and forming the initial flock at approximately $t=21$ seconds; tails show $8$ seconds of trajectory history.}
\label{fig:swarming1}
\end{figure}

\begin{figure}[ht]
\centering
\includegraphics[width=\linewidth]{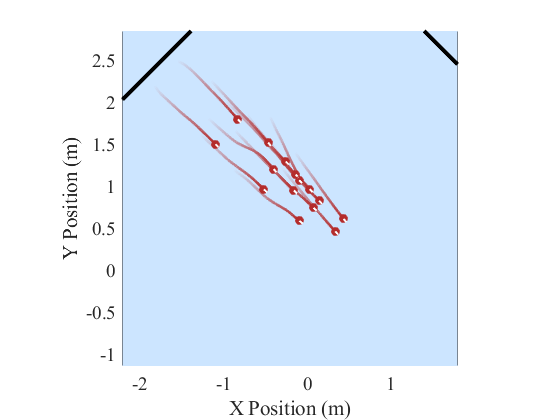}
\caption{Boids cruising to the south-east at approximately $t=85$ seconds after reaching the north-west wall and changing direction; tails show $8$ seconds of trajectory history.}
\label{fig:swarming2}
\end{figure}

Next, we introduce a simple predator model.
The data in \cite{Pitcher1983Predator-avoidanceSplit} implies that individual sand-eels treat predators as a moving obstacles.
In fact, they explicitly state that ``... the mackerel ate very few of the sand-eels throughout the duration of the experiment ...''\textemdash implying that the predator avoidance behavior ought to emerge without an antagonistic predator model.
With this justification, our predator follows a simple rule: orient toward the center of the boid flock and travel in a straight line for $8$ seconds.
The predator moves $20\%$ faster than the boids, and as such it is able to pass through the swarm and influence its behavior.
We found that $8$ seconds was a reasonable tradeoff to have the predator make several passes through the swarm without requiring significantly more simulation time.
As with the previous simulation, the flock quickly formed and began cruising across the domain.
The predator made multiple passes through the swarm, and each time the boids avoided the predator and quickly reformed.
A simulation snapshot is presented in Fig. \ref{fig:vacuole} near $t=52$ s, where the boids qualitatively exhibit the vacuole behavior seen in the sand-eel experiments \cite{Pitcher1983Predator-avoidanceSplit}.

\begin{figure}[ht]
    \centering
    \includegraphics[width=0.8\linewidth]{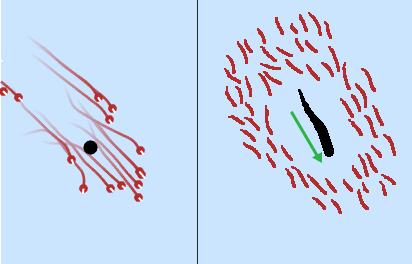}
    \caption{Left: apparent vacuole behavior exhibited by the boids the predator approaches from behind. Right: vacuole behavior observed in sand-eels, recreated from \cite{Pitcher1983Predator-avoidanceSplit}.}
    \label{fig:vacuole}
\end{figure}

Finally, we saved the size of each boids' neighborhood (Definition \ref{def:neighborhood}) at each time instant throughout the simulation.
A histogram of neighborhood size is given in Fig. \ref{fig:neighborhood} for the simulation containing the predator.
The distribution of neighborhood sizes is approximately Weibull, with $4$ neighbors being the most frequent.
This supports existing results in the biology literature \cite{Ballerini2008InteractionStudy}, which claims that only considering $3$--$5$ neighbors may be optimal for predator avoidance in 2D swarms.

\begin{figure}[ht]
    \centering
    \includegraphics[width=0.85\linewidth]{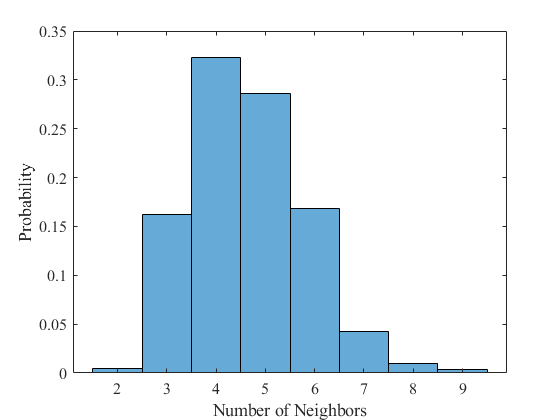}
    \caption{Neighborhood size histogram for $N = 15$ boids during the $120$ second simulation with a predator.}
    \label{fig:neighborhood}
\end{figure}

\section{Conclusion} \label{sec:conclusion}

In this letter, we constructed a decentralized control policy to generate emergent swarming behavior for boids operating in a constrained environment.
We extended current ecologically-inspired approaches beyond control minimization and instead considered an optimal speed.
We rigorously linked our event-triggered scheme for constraint relaxation to a finite state machine, which guarantees recursive feasibility without the use of slack variables.
To verify the emergence of swarming behavior, we performed two simulations; one with no predator, and the second with a velocity obstacle that tracks the centroid of the flock.

Future work includes extending our analysis to $\mathbb{R}^3$ with explicit collision avoidance constraints.
Further exploring the distribution of neighborhood size for Voronoi neighborhoods is another compelling direction\textemdash particularly whether these distributions only coincidentally agree with the optimal neighborhood sizes found in the biology literature.
Finally, experiments to replicate swarming behavior with physical robots will likely yield valuable insights.

\bibliographystyle{IEEEtran}
\bibliography{mendeley, IDS_Pubs}

\end{document}